\documentclass[oneside,reqno,american]{amsart}
\usepackage[T1]{fontenc}
\usepackage[utf8]{inputenc}
\usepackage{babel}
\usepackage{prettyref}
\usepackage{float}
\usepackage{mathrsfs}
\usepackage{mathtools}
\usepackage{amsthm}
\usepackage{amssymb}
\usepackage[pdfusetitle,
 bookmarks=true,bookmarksnumbered=false,bookmarksopen=false,
 breaklinks=false,pdfborder={0 0 1},backref=false,colorlinks=false]
 {hyperref}
\hypersetup{
 colorlinks=true,citecolor=blue,linkcolor=blue,linktocpage=true}

\makeatletter
\numberwithin{equation}{section}
\numberwithin{figure}{section}
\theoremstyle{plain}
\newtheorem{thm}{\protect\theoremname}
\theoremstyle{definition}
\newtheorem{defn}[thm]{\protect\definitionname}
\theoremstyle{plain}
\newtheorem{cor}[thm]{\protect\corollaryname}
\theoremstyle{remark}
\newtheorem{rem}[thm]{\protect\remarkname}

\@ifundefined{date}{}{\date{}}
\usepackage{tikz}
\usetikzlibrary{arrows.meta, positioning}

\newrefformat{cor}{Corollary~\ref{#1}}
\newrefformat{subsec}{Section~\ref{#1}}
\newrefformat{lem}{Lemma~\ref{#1}}
\newrefformat{thm}{Theorem~\ref{#1}}
\newrefformat{sec}{Section~\ref{#1}}
\newrefformat{chap}{Chapter~\ref{#1}}
\newrefformat{prop}{Proposition~\ref{#1}}
\newrefformat{exa}{Example~\ref{#1}}
\newrefformat{tab}{Table~\ref{#1}}
\newrefformat{rem}{Remark~\ref{#1}}
\newrefformat{def}{Definition~\ref{#1}}
\newrefformat{fig}{Figure~\ref{#1}}
\newrefformat{claim}{Claim~\ref{#1}}

\makeatother

\providecommand{\corollaryname}{Corollary}
\providecommand{\definitionname}{Definition}
\providecommand{\remarkname}{Remark}
\providecommand{\theoremname}{Theorem}

\begin{document}
\title[Factorization of positive definite kernels]{Factorization of positive definite kernels. Correspondences: $C^{*}$-algebraic
and operator valued context vs scalar valued kernels}
\author{Palle E.T. Jorgensen}
\address{(Palle E.T. Jorgensen) Department of Mathematics, The University of
Iowa, Iowa City, IA 52242-1419, U.S.A.}
\email{palle-jorgensen@uiowa.edu}
\author{James Tian}
\address{(James F. Tian) Mathematical Reviews, 535 W William St, Ste 210, Ann
Arbor, MI 48103-4978, U.S.A.}
\email{jft@ams.org}
\begin{abstract}
We introduce and study a class $\mathcal{M}$ of generalized positive
definite kernels of the form $K\colon X\times X\to L(\mathfrak{A},L(H))$,
where $\mathfrak{A}$ is a unital $C^{*}$-algebra and $H$ a Hilbert
space. These kernels encode operator-valued correlations governed
by the algebraic structure of $\mathfrak{A}$, and generalize classical
scalar-valued positive definite kernels, completely positive (CP)
maps, and states on $C^{*}$-algebras. Our approach is based on a
scalar-valued kernel $\tilde{K}\colon(X\times\mathfrak{A}\times H)^{2}\to\mathbb{C}$
associated to $K$, which defines a reproducing kernel Hilbert space
(RKHS) and enables a concrete, representation-theoretic analysis of
the structure of such kernels. We show that every $K\in\mathcal{M}$
admits a Stinespring-type factorization $K(s,t)(a)=V(s)^{*}\pi(a)V(t)$.
In analogy with the Radon--Nikodym theory for CP maps, we characterize
kernel domination $K\leq L$ in terms of a positive operator $A\in\pi_{L}(\mathfrak{A})'$
satisfying $K(s,t)(a)=V_{L}(s)^{*}\pi_{L}(a)AV_{L}(t)$. We further
show that when $\pi_{L}$ is irreducible, domination implies scalar
proportionality, thus recovering the classical correspondence between
pure states and irreducible representations.
\end{abstract}

\maketitle

\section{Introduction}

The theory of positive definite (p.d.) kernels plays a foundational
role in the analysis of operator-valued functions, dilation theory,
and noncommutative function theory, with deep connections to completely
positive (CP) maps, reproducing kernel Hilbert spaces (RKHSs), and
operator algebras. In the setting where the kernel takes values in
operator spaces over a $C^{*}$-algebra, understanding factorization,
comparison, and structure of such kernels is central to applications
ranging from quantum probability to machine learning with operator-valued
data \cite{MR4295177,MR4561157,MR4760560,MR4787903,MR4867336,MR4887444}.

One of the central tools in this theory is the Stinespring Dilation
Theorem, a structural result that characterizes completely positive
maps $\Phi\colon\mathfrak{A}\to B(H)$ as compressions of $*$-representations:
there exists a Hilbert space $\mathscr{K}$, a $*$-representation
$\pi\colon\mathfrak{A}\to L(\mathscr{K})$, and a bounded linear operator
$V\colon H\to\mathscr{K}$ such that $\Phi(a)=V^{*}\pi(a)V$, for
all $a\in\mathfrak{A}$. This result provides a canonical form for
CP maps, and serves as a foundation for much of modern operator theory
and quantum functional analysis \cite{MR69403,MR102761,MR2336315}.
In quantum information theory, Stinespring's theorem underpins the
mathematical formulation of quantum operations. A general quantum
channel---such as an open-system evolution or noisy measurement process---is
modeled by a completely positive trace-preserving (CPTP) map. Stinespring's
dilation shows that such maps are, in essence, restrictions of unitary
evolutions on a larger space, followed by partial traces over environmental
degrees of freedom. This is essential in understanding decoherence,
quantum noise, and reversible simulation of quantum dynamics, and
forms the core of both theoretical and applied quantum computing.

In this paper, we develop a systematic framework for a class $\mathcal{M}$
of generalized positive definite kernels $K\colon X\times X\to L(\mathfrak{A},L(H))$,
where $\mathfrak{A}$ is a unital $C^{*}$-algebra and $H$ is a Hilbert
space. Our approach is distinguished by its reliance on a scalar-valued
kernel---that is, we embed the kernel structure through scalar-valued
inner products, avoiding the need to construct or analyze vector-valued
RKHSs directly. This scalar-valued machinery yields a more transparent
view of dilation and domination properties, providing new insights
into factorization, comparison, and Radon--Nikodym derivatives of
such kernels.

In fact each $K\in\mathcal{M}$ may be interpreted as a generalized
state on the $C^{*}$-algebra $\mathfrak{A}$. While a classical state
is a positive linear functional $\omega\colon\mathfrak{A}\to\mathbb{C}$
satisfying $\omega(1_{\mathfrak{A}})=1$, the present setting replaces
the scalar target with $L(H)$-valued p.d. kernels. That is, $\Phi\left(a\right)\left(s,t\right)\coloneqq K\left(s,t\right)\left(a\right)$
defines, for each fixed $a$, a p.d. kernel on $X\times X$, taking
values in $L(H)$, satisfying a certain positivity condition, which
generalizes the Gelfand--Naimark--Segal (GNS) and CP constructions. 

The main contributions include a full characterization of kernels
in the class $\mathcal{M}$ via Stinespring-type factorizations, and
a necessary and sufficient condition for kernel domination using positive
operators in the commutant $\pi_{L}(\mathfrak{A})'$, extending results
from \cite{MR4574232,MR4697457,MR4887444}. When the representing
$\mathfrak{A}$-module is minimal and the representation $\pi_{L}$
is irreducible, we show that kernel domination implies scalar proportionality,
paralleling results on states of $C^{*}$-algebras \cite{MR2743416,MR1468230,MR442701,MR512360,MR660825,MR1728123}.

This perspective unifies and extends prior work on operator-valued
Gaussian processes (see, e.g., \cite{MR4787903}), completely positive
semigroups \cite{MR1846747}, and functional models of noncommutative
functions \cite{MR4697457,MR4181326}. In particular, it complements
recent kernel-theoretic dualities in learning theory and dynamical
systems by supplying a precise algebraic and operator-theoretic structure
to the dominance structure of such kernels.

Our results build on and generalize classical dilation and representation
theorems, especially those due to Stinespring \cite{MR102761,MR2366366,MR69403},
Arveson \cite{MR1728123,MR2329688}, and more recent developments
in noncommutative kernel theory \cite{MR4574232,MR2383513,MR2336315}.
They also parallel the lines of research investigating hereditary
kernels and Arveson-type extensions \cite{MR4574232}, and establish
a scalar-valued analytic foundation for further exploration of domination
theory and Radon--Nikodym derivatives in the operator-algebraic context. 

\section{The class $\mathcal{M}$ of generalized p.d. kernels}

For our analysis, it will be convenient that we adopt the physics
convention for Hilbert inner products to be linear in the second variable.
\begin{defn}
\label{def:1}Let $X$ be a set, $H$ a Hilbert space, and $\mathfrak{A}$
a $C^{*}$-algebra with unit $1_{\mathfrak{A}}$. Denote by $\mathcal{M}$
the class of generalized positive definite (p.d.) kernels 
\[
K\colon X\times X\rightarrow L\left(\mathfrak{A},L\left(H\right)\right)
\]
such that 
\begin{equation}
\sum_{i,j=1}^{N}\left\langle u_{i},K\left(s_{i},s_{j}\right)\left(a_{i}^{*}a_{j}\right)u_{j}\right\rangle _{H}\geq0\label{eq:a1}
\end{equation}
for all $\left(u_{i}\right)_{i=1}^{N}\subset H$, $\left(s_{i}\right)_{i=1}^{N}\subset X$,
$\left(a_{i}\right)_{i=1}^{N}\subset\mathfrak{A}$, and all $N\in\mathbb{N}$.
\end{defn}

For each $K$ satisfying \eqref{eq:a1}, there is a naturally associated
scalar-valued p.d. kernel $\tilde{K}\colon\Omega\times\Omega\rightarrow\mathbb{C}$,
where 
\[
\Omega\coloneqq X\times\mathfrak{A}\times H
\]
and 
\begin{equation}
\tilde{K}\left(\left(s,a,u\right),\left(t,b,v\right)\right)\coloneqq\left\langle u,K\left(s,t\right)\left(a^{*}b\right)v\right\rangle _{H}.\label{eq:2}
\end{equation}
Let $H_{\tilde{K}}$ be the corresponding reproducing kernel Hilbert
space (RKHS). Recall that 
\begin{equation}
H_{\tilde{K}}=\overline{span}\left\{ \tilde{K}_{\left(t,b,v\right)}:t\in X,b\in\mathfrak{A},v\in H\right\} \label{eq:a3}
\end{equation}
where 
\begin{equation}
\tilde{K}_{\left(t,b,v\right)}\colon\Omega\rightarrow\mathbb{C},\quad\tilde{K}_{\left(t,b,v\right)}=\left\langle \cdot,K\left(\cdot,t\right)\left(\cdot b\right)v\right\rangle _{H}.\label{eq:a4}
\end{equation}

This construction generalizes the classical theory of reproducing
kernels as developed by Aronszajn \cite{MR24065}, and is further
elaborated in modern treatments such as \cite{MR3526117,MR2975345}.
Kernel methods have since found widespread application in operator
theory, statistics, inverse problems, and machine learning \cite{MR4860620,MR4905634,MR4595463,MR4693231,MR4816645}.
Application of our present setting is illustrated in the following
diagram:

\begin{figure}[H]
\begin{center}
\begin{tikzpicture}[
  every node/.style={font=\small, align=center},
  box/.style={draw, rounded corners, minimum width=3.4cm, minimum height=1.1cm},
  arrow/.style={-Stealth, thick}
]

\node[draw, circle, minimum size=2.1cm] (center) {Generalized\\PD Kernels};

\node[box, above=0.5cm of center] (gp) {Operator-Valued\\Gaussian Processes};
\node[box, right=1cm of center] (quantum) {Quantum Channels \&\\CP Maps};
\node[box, below=0.5cm of center] (ml) {Kernel-Based\\Machine Learning};
\node[box, left=1cm of center] (rkhs) {Reproducing Kernel\\Hilbert Spaces (RKHS)};

\draw[arrow, bend left=20] (gp) to (quantum);
\draw[arrow, bend left=20] (quantum) to (ml);
\draw[arrow, bend left=20] (ml) to (rkhs);
\draw[arrow, bend left=20] (rkhs) to (gp);

\end{tikzpicture}
\end{center}

\caption{Illustrations of related applications of the generalized kernels.}

\end{figure}
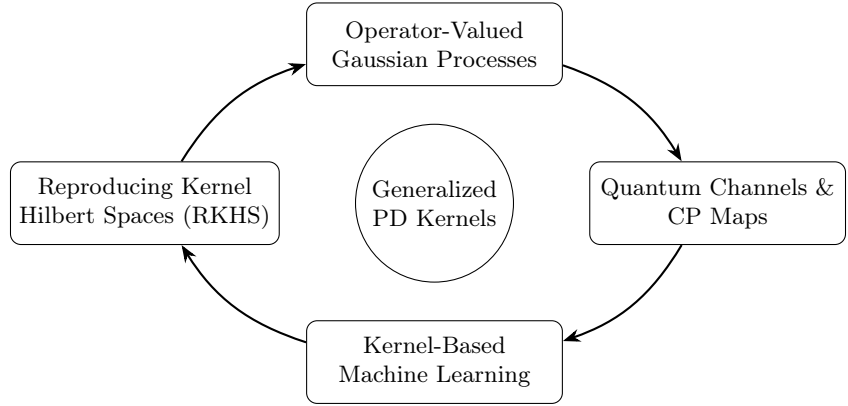

\begin{thm}
A kernel $K$ satisfies \eqref{eq:a1} if and only if it factors as
\[
K\left(s,t\right)\left(a\right)=V\left(s\right)^{*}\pi\left(a\right)V\left(t\right),
\]
for all $s,t\in X$ and $a\in\mathfrak{A}$, where $V\left(t\right)\colon H\rightarrow\mathscr{L}$
is an operator from $H$ into some Hilbert space $\mathscr{L}$, and
$\pi\colon\mathfrak{A}\rightarrow L\left(\mathscr{L}\right)$ is a
representation of $\mathfrak{A}$. 

Moreover, if $\mathscr{L}$ is minimal, i.e., 
\[
\mathscr{L}=\overline{span}\left\{ \pi\left(a\right)V\left(t\right)u:a\in\mathfrak{A},t\in X,u\in H\right\} 
\]
then $\mathscr{L}\simeq H_{\tilde{K}}$, and so $\mathscr{L}$ is
unique up to unitary equivalence.
\end{thm}

\begin{proof}
Let $K$, $\tilde{K}$ and $H_{\tilde{K}}$ be as above. For each
$t\in X$, let $V\left(t\right)\colon H\rightarrow H_{\tilde{K}}$
be the operator 
\[
v\mapsto V\left(t\right)v=\tilde{K}_{\left(t,1_{\mathfrak{A}},v\right)}.
\]
By \eqref{eq:a4}, $V\left(t\right)$ is linear and 
\begin{eqnarray*}
\left\Vert V\left(t\right)v\right\Vert _{H}^{2} & = & \left\langle \tilde{K}_{\left(t,1_{\mathfrak{A}},v\right)},\tilde{K}_{\left(t,1_{\mathfrak{A}},v\right)}\right\rangle _{H}\\
 & \underset{\left(\ref{eq:2}\right)}{=} & \left\langle v,K\left(t,t\right)\left(1_{\mathfrak{A}}\right)v\right\rangle _{H}<\infty.
\end{eqnarray*}

Since 
\begin{eqnarray*}
\left\langle V\left(t\right)v,\tilde{K}_{\left(s,a,u\right)}\right\rangle _{H_{\tilde{K}}} & = & \left\langle \tilde{K}_{\left(t,1_{\mathfrak{A}},v\right)},\tilde{K}_{\left(s,a,u\right)}\right\rangle _{H_{\tilde{K}}}\\
 & = & \left\langle v,K\left(t,s\right)\left(a\right)u\right\rangle _{H}
\end{eqnarray*}
and $\left\{ \pi\left(a\right)V\left(t\right)u:a\in\mathfrak{A},t\in X,u\in H\right\} $
is dense in $H_{\tilde{K}}$ (by definition \eqref{eq:a3}), it follows
that 
\begin{equation}
V\left(t\right)^{*}\tilde{K}_{\left(s,a,u\right)}=K\left(t,s\right)\left(a\right)u\label{eq:4}
\end{equation}
for all $s,t\in X$, $a\in\mathfrak{A}$, $u\in H$.

Let $\pi\colon\mathfrak{A}\rightarrow H_{\tilde{K}}$ be the map 
\begin{equation}
\pi\left(a\right)\tilde{K}_{\left(t,b,v\right)}=\tilde{K}_{\left(t,ab,v\right)},\label{eq:5}
\end{equation}
for all $t\in X$, $a,b\in\mathfrak{A}$, $v\in H$, extended to all
$H_{\tilde{K}}$. It is clear that $\pi$ is a representation of $\mathfrak{A}$
in $H_{\tilde{K}}$. Indeed, 
\begin{align*}
\left\langle \tilde{K}_{\left(s,c,u\right)},\pi\left(a\right)\tilde{K}_{\left(t,b,v\right)}\right\rangle _{H_{\tilde{K}}} & =\left\langle u,K\left(s,t\right)\left(c^{*}ab\right)v\right\rangle _{H}\\
 & =\left\langle u,K\left(s,t\right)\left(\left(a^{*}c\right)^{*}b\right)v\right\rangle _{H}\\
 & =\left\langle \tilde{K}_{\left(s,a^{*}c,u\right)},\tilde{K}_{\left(t,b,v\right)}\right\rangle _{H_{\tilde{K}}}\\
 & =\left\langle \pi\left(a^{*}\right)\tilde{K}_{\left(s,c,u\right)},\tilde{K}_{\left(t,b,v\right)}\right\rangle _{H_{\tilde{K}}}
\end{align*}
so that $\pi\left(a\right)^{*}=\pi\left(a^{*}\right)$. 

Finally, one checks that 
\begin{eqnarray*}
V\left(s\right)^{*}\pi\left(a\right)V\left(t\right)u & = & V\left(s\right)^{*}\pi\left(a\right)\tilde{K}_{\left(t,1_{\mathfrak{A}},v\right)}\\
 & \underset{\left(\ref{eq:5}\right)}{=} & V\left(s\right)^{*}\tilde{K}_{\left(t,a,v\right)}\\
 & \underset{\left(\ref{eq:4}\right)}{=} & K\left(s,t\right)\left(a\right)v.
\end{eqnarray*}

The rest of the theorem follows from standard arguments in the literature,
and the details are omitted here. 
\end{proof}

\section{A partial order on operator-kernel spaces, and its use in an analysis
of families of non-commuting kernels}

The set $\mathcal{M}$ carries a natural partial order: 
\begin{defn}
For two kernels $K,L$ in $\mathcal{M}$, it is said that $L$ dominates
$K$, or $K\leq L$, if $L-K\in\mathcal{M}$. 
\end{defn}

\begin{thm}
\label{thm:4}Let $K,L\in\mathcal{M}$, such that
\begin{align*}
K\left(s,t\right)\left(a\right) & =V_{K}^{*}\left(s\right)\pi_{K}\left(a\right)V_{K}\left(t\right)\\
L\left(s,t\right)\left(a\right) & =V_{L}^{*}\left(s\right)\pi_{L}\left(a\right)V_{L}\left(t\right).
\end{align*}
Then $K\leq L$ if and only if there is a positive operator $A\colon H_{\tilde{L}}\rightarrow H_{\tilde{L}}$,
$0\leq A\leq I_{H_{\tilde{L}}}$, such that $A\in\pi_{L}\left(\mathfrak{A}\right)'$,
and 
\begin{equation}
K\left(s,t\right)\left(a\right)=V_{L}^{*}\left(s\right)\pi_{L}\left(a\right)AV_{L}\left(t\right).\label{eq:7}
\end{equation}
\end{thm}

\begin{proof}
The identify \eqref{eq:7} follows from a basic fact for scalar-valued
kernels (see, e.g., \cite{MR0051437,MR24065}). Recall that for $K,L:X\times X\rightarrow\mathbb{C}$
p.d., $K\leq L\Longleftrightarrow L-K\geq0$, which is equivalent
to $H_{K}\subset H_{L}$, where the containment of the respective
RKHSs is contractive. In particular, the map $T\colon L\left(\cdot,t\right)\mapsto K\left(\cdot,t\right)\in H_{K}\subset H_{L}$
extends by density to a positive operator on $H_{L}$, with $0\leq T\leq I_{H_{L}}$. 

Now apply this to $\tilde{K}$ and $\tilde{L}$: By assumption, $\tilde{K}\leq\tilde{L}$
as scalar-valued kernels. Therefore, the map
\[
T:\tilde{L}_{\left(t,a,v\right)}\mapsto\tilde{K}_{\left(t,a,v\right)}\in H_{\tilde{K}}\subset H_{\tilde{L}}
\]
extends to a positive operator in $H_{\tilde{L}}$ with $0\leq T\leq I_{H_{\tilde{L}}}$.
Note that $\pi_{K}=\pi_{L}\big|_{H_{\tilde{K}}}$, and so 
\begin{align*}
T\pi_{L}\left(a\right)V_{L}\left(t\right)v & =T\tilde{L}_{\left(t,a,v\right)}=\tilde{K}_{\left(t,a,v\right)}\\
 & =\pi_{K}\left(a\right)V_{K}\left(t\right)v=\pi_{L}\left(a\right)TV_{L}\left(t\right)v.
\end{align*}
This implies that $T\in\pi_{L}\left(\mathfrak{A}\right)'$. Then,
\begin{align*}
K\left(s,t\right)\left(a,b\right) & =\left(\pi_{K}\left(a\right)V_{K}\left(s\right)\right)^{*}\left(\pi_{K}\left(b\right)V_{K}\left(t\right)\right)\\
 & =\left(T\pi_{L}\left(a\right)V_{L}\left(t\right)\right)^{*}\left(T\pi_{L}\left(b\right)V_{L}\left(t\right)\right)\\
 & =V_{L}\left(t\right)^{*}T^{2}\pi\left(a^{*}b\right)V_{L}\left(t\right),
\end{align*}
with $A\coloneqq T^{2}$, and this is \eqref{eq:7}. 
\end{proof}
\begin{defn}
In the case of \eqref{eq:7}, $A$ is said to be the Radon-Nikodym
derivative of $K$ with respect to $L$, denoted by 
\[
A=dK/dL.
\]
 
\end{defn}

\begin{cor}
The following are equivalent: 
\begin{enumerate}
\item $\pi_{L}$ is irreducible.
\item $K\leq L$ $\Longleftrightarrow$ $K=\lambda L$ for some constant
$\lambda\in\mathbb{C}$. 
\item $dK/dL=\lambda I_{H_{\tilde{L}}}$ for some constant $\lambda\in\mathbb{C}$. 
\end{enumerate}
\end{cor}

\begin{rem}
This corollary closely mirrors the classical correspondence in the
Gelfand--Naimark--Segal (GNS) construction, where pure states on
a $C^{*}$-algebra correspond to irreducible representations. In our
setting, the map
\[
\Phi_{K}\left(a\right):X\times X\rightarrow L\left(H\right),\quad\Phi_{K}\left(a\right)\coloneqq K\left(s,t\right)\left(a\right)
\]
where $K\in\mathcal{M}$, generalize classical states by taking values
not in scalars, but in $L(H)$-valued positive definite kernels. 

Historically, this conceptual shift dates back to the early 1950s,
when Segal asked Stinespring to extend the GNS construction. Stinespring's
answer was his famous dilation theorem, which replaced scalar-valued
states with completely positive (CP) maps into operator algebras.
In this way, the passage from classical states to operator-valued
kernels such as $K$ may be viewed as a continuation of this line
of thought---trading scalar positivity for operator positivity, and
classical irreducibility for that of the associated representation
$\pi_{K}$. In other words, $\Phi_{K}$ is a pure state (in the generalized
sense) if and only if $\pi_{K}$ is irreducible. 

This analogy extends further. Every state on a unital $C^{*}$-algebra
is a Choquet integral applied to the pure states. The non-pure states
are called mixed states. In the context of quantum mechanics, pure
states represent a system with complete knowledge, while mixed states
represent a system with some degree of uncertainty or incomplete information.
The Choquet boundary captures the idea that pure states are the extreme
points of the convex set of all possible quantum states; see, e.g.,
\cite{MR1151615,MR660825,MR149258,MR729530}. In our generalized setting,
irreducibility of the associated representation $\pi_{K}$ plays the
role of extremality, imitating the structure of classical state space
decompositions.

For the benefit of readers, we include the following relevant citations
on operator algebras, \cite{MR1276162,MR512360,MR442701,MR1468230}.
\end{rem}

\begin{cor}[Quantum Channel Simulation]
\label{cor:8}Let $K,L\in\mathcal{M}$ be positive definite kernels
associated to quantum systems, with factorizations:
\[
K(s,t)(a)=V_{K}(s)^{*}\pi_{K}(a)V_{K}(t),\quad L(s,t)(a)=V_{L}(s)^{*}\pi_{L}(a)V_{L}(t),
\]
where $\pi_{K},\pi_{L}$ are $*$-representations of a $C^{*}$-algebra
$\mathfrak{A}$, and the factorizations are minimal.

Suppose $K\leq L$, i.e., the kernel $K$ is dominated by $L$ via:
\[
K(s,t)(a)=V_{L}(s)^{*}\pi_{L}(a)AV_{L}(t),
\]
for some $0\leq A\leq I_{H_{\tilde{L}}}$ in $\pi_{L}(\mathfrak{A})'$.

Then there exists a completely positive trace-nonincreasing map (i.e.,
a quantum operation)
\[
\Phi:\pi_{L}(\mathfrak{A})\to\pi_{L}(\mathfrak{A})
\]
such that
\[
K(s,t)(a)=V_{L}(s)^{*}\Phi(\pi_{L}(a))V_{L}(t),
\]
where $\Phi$ is of the form:
\[
\Phi(T)=A^{1/2}TA^{1/2},\quad T\in\pi_{L}(\mathfrak{A}),
\]
which corresponds to post-processing by a quantum effect (POVM element)
$A$.
\end{cor}

\begin{rem}
The proof of \prettyref{cor:8} follows immediate from \prettyref{thm:4}.
There is a rich literature on positive operator valued measures (POVM).
Of special relevance is \cite{MR4862902,MR4557727,MR4555558,MR4760560}.

In the language of quantum information, the kernel $K$ may be viewed
as modeling a quantum process---such as a generalized quantum channel---that
can be simulated or degraded from the process described by $L$. The
operator $A\in\pi_{L}(\mathfrak{A})'$, which mediates the dominance
relation $K\leq L$, plays the role of a quantum effect or filter
in the Heisenberg picture, transforming observables via a completely
positive trace-nonincreasing map. This structure mirrors the notion
of quantum sufficiency, wherein one statistical experiment (or quantum
measurement model) is derivable from another via post-processing.
Notably, if $\pi_{L}$ is irreducible, then $A$ must lie in $\mathbb{C}\cdot I$,
implying that $K\leq L$ entails a scalar attenuation of the kernel---resembling
a fidelity degradation between quantum operations.

More precisely, the dependence of $V_{L}(t)\colon H\to\mathscr{L}$
on the parameter $t\in X$ reflects a family of quantum states or
embeddings labeled by classical data. Each $V_{L}(t)$ can be interpreted
as preparing the system in a quantum state associated with input $t$.
The kernel expression
\[
K(s,t)(a)=V_{L}(s)^{*}\,\pi_{L}(a)\,A\,V_{L}(t)
\]
then represents a generalized transition amplitude: prepare the system
at $t$, apply observable $\pi_{L}(a)$, attenuate with $A$, and
compare with the preparation at $s$. The operator $A$ acts as a
quantum effect or post-processing filter. This structure appears in
quantum learning and quantum instruments, where classical inputs are
encoded into quantum states and kernels track correlations or similarities
between them through observables. The variation in $V_{L}(t)$ encodes
how classical parameters affect quantum behavior.

The non-commutative framework forming the setting for our present
factorization results is both natural; and it is also motivated by
questions that are central in the context of quantum information,
e.g., in consideration of open quantum systems. Indeed, in the recent
literature in quantum information, there is a variety of such new
developments. Full details will be beyond the scope of our present
paper, and so for our present purpose, we limit ourselves to the following
citations to the literature \cite{MR4760335,MR4619365,MR3390587,MR4719025,MR4230450}.
\end{rem}

\begin{rem}
One may replace the $C^{*}$-algebra $\mathfrak{A}$ in \prettyref{def:1}
by a group algebra $\mathbb{C}[G]$, then all essential ingredients
from the above general $C^{*}$-algebraic framework carry over: positivity,
sesquilinearity, factorization and minimality are preserved. Note
that the positive-definiteness assumption \eqref{eq:a1} becomes:
\[
\sum_{i,j=1}^{N}\left\langle u_{i},K\left(s_{i},s_{j}\right)\left(g_{i}^{-1}g_{j}\right)u_{j}\right\rangle _{H}\geq0
\]
for all $\left(u_{i}\right)_{i=1}^{N}\subset H,\left(s_{i}\right)_{i=1}^{N}\subset X,\left(g_{i}\right)_{i=1}^{N}\subset G$,
and all $N\in\mathbb{N}.$ 

The group setting provides a concrete and symmetric structure that
is particularly natural in harmonic analysis, representation theory,
and quantum information, where groups often encode symmetries, conserved
quantities, or control operations. The kernel factorization thus becomes
a bridge between abstract operator algebra methods and classical constructions
involving group actions and unitary dynamics.
\end{rem}

\bibliographystyle{amsalpha}
\bibliography{ref}

\end{document}